\documentclass[11pt,twoside]{article}
\usepackage{amsmath, amssymb, amsfonts, amstext, amsthm, textcomp, enumerate}
\usepackage{tikz}
\usepackage{comment}
\usepackage[mathscr]{euscript}
\usepackage{float}
\usepackage{booktabs}
\usepackage{mathtools}
\usepackage{tabularx}
\usepackage[left=20mm,top=0.5in,bottom=15mm]{geometry}
\usepackage{graphicx}
\usepackage{caption}
\usepackage{epstopdf}
\usepackage{longtable}
\usepackage[utf8]{inputenc}
\usepackage{color}
\usepackage{hyperref}
\usepackage{graphicx}
\usepackage{dcolumn}
\usepackage{bm}
\usepackage{epstopdf}
\usepackage[english]{babel}
\usepackage{subfigure}
\usepackage{color}

\usepackage{ulem}
\newtheorem{thm}{Theorem}[section]
\newtheorem{lem}[thm]{Lemma}
\newtheorem{cor}[thm]{Corollary}

\newtheorem{rem}[thm]{Remark}
\newtheorem{ex}[thm]{Example}

\newcommand{\mnr}{\mathbf{M}_n\,(\mathbb{R})}

\newfont{\bb}{msbm10}

\title{Graphs of Moore-Penrose inverse of matrices possessing the treeangle property}
\author{A. M. Encinas
\thanks{Department de Mathem\`atiques, Universitat Polit\`ecnica de Catalunya, 08034 Barcelona, Spain\\ (andres.marcos.encinas@upc.edu).}  $\cdot$ 
K. Kranthi Priya \thanks{Department of Mathematics, Indian Institute of 
		Technology Madras,
		Chennai 600036, India (ma21d010@smail.iitm.ac.in, kcskumar@iitm.ac.in).}  $\cdot$  K. C. Sivakumar $^\dagger$ }
\begin{document}
\maketitle

\begin{abstract}
It is known that the inverse of an invertible real square matrix satisfying the treeangle property, is a treediagonal matrix. A converse statement also holds. 
We show that the verbatim analogues are not true for the Moore-Penrose inverse, and obtain the precise structure of graphs corresponding to the Moore-Penrose inverse
of matrices possessing the treeangle property.

\end{abstract}

\vskip.25in
\textit{Keywords:} Treeangle property, treediagonal matrix, 
Moore-Penrose inverse.\\

\textit{AMS Subject classifications:} 15A03, 15A09. 

\newpage
\section{Introduction} 

Let $\Gamma$ be a graph with the vertex set $V(\Gamma)=V:=\{v_1,v_2,\ldots, v_n\}$, where, $v_i$'s are distinct from each other, and the edge set  
\[
E(\Gamma) = \{\{v_i, v_j\} :i \neq j,~ v_i \sim v_j\},
\]
where the notation $v_i \sim v_j$ indicates that the vertices $v_i$ and $v_j$ are {\it adjacent} in $\Gamma$, i.e., there is an edge between $v_i$ and $v_j$. Therefore, we are assuming that $\Gamma$ is an (undirected) {\it simple} graph, {\it i.e.} has no multiple edges between two given vertices and has no loops around any vertex. For a given vertex $v$, let $N(v),$ the {\it neighbourhood of $v$}, denote the set of vertices which are adjacent to $v$, and $d(v):=|N(v)|$ (the cardinality) denote the {\it degree of $v.$} If $d(v)=1$ then $v$ is said to be a {\it pendant} vertex, while $v$ is called an {\it interior} vertex, if $d(v)>1$.

 A {\it path of length $m$} in $\Gamma$  is a sequence $v_{i_1},v_{i_2},\ldots,v_{i_m},v_{i_{m+1}}$ of vertices, distinct from each other, such that $v_{i_j}\sim v_{i_{j+1}}$ for all $j=1,\ldots,m$. A {\it cycle of length $m$} in $\Gamma$  is a sequence $v_{i_1},v_{i_2},\ldots,v_{i_{m-1}},v_{i_m}$ of vertices, distinct from each other, such that $v_{i_j}\sim v_{i_{j+1}}$ for all $j=1,\ldots,m-1$ and $v_{i_m}\sim v_{i_1}$. 
 
If $\Gamma$ is connected and acyclic, then $\Gamma$ is a {\it tree}. Therefore, a graph is a tree if and only if there exists a unique path between any two vertices $u, v \in V$. It is well-known that any tree has at least two pendant vertices.

Let us recall the operation of deleting a vertex from a graph $\Gamma$. Removal of a vertex $v$ means that we remove $v$ from the vertex set $V$ and all edges that are incident to $v$ from  $\Gamma.$ The resultant graph is denoted by $\Gamma \setminus \{v\}$. Similarly, the operation  of deleting an edge $e=\{u,v\}$,  means simply removing the edge $e$ from the graph $\Gamma$. The resultant graph is denoted by $\Gamma \setminus \{e\}.$

Let $\mathbf{M}_{m\times n}(\mathbb{R})$ denote the set of all real matrices of size $m\times n$. When $m=n,$ we simply write $\mnr.$ 

For a given tree $\Gamma$, we define a matrix $A=(a_{ij}) \in \mnr$ to be {\it treediagonal} (more precisely $\Gamma$-treediagonal) if $a_{ij}=a_{ji}=0$ for $i\neq j$ and $\{v_i,v_j\} \notin E.$  Additionally, when $\{v_i,v_j\}\in E$, we assume that at least one of $a_{ij}$ or $a_{ji}$ is nonzero. 
 Recall that $A\in  \mnr$ is said to be {\it reducible} if there is a permutation matrix $P\in  \mnr$ such that $P^TAP$ has the form $$\begin{pmatrix}
       B_{11} & B_{12}\\
       0 & B_{22}
   \end{pmatrix}$$ for some square matrices $B_{11}$ and $B_{22}$ of order at least one. A matrix is {\it irreducible} if it is not reducible.
Thus, a $\Gamma$-treediagonal matrix may be reducible. Observe that, if $P_n$ denotes the path on $n$ vertices, then any tridiagonal matrix $A=(a_{ij})$ of order $n$ is a $P_n$-treediagonal matrix if $|a_{i,i+1}|+|a_{i+1,i}|>0$ for $i=1,\ldots,n-1$.

For a given matrix $A=(a_{ij}) \in \mnr$, we define the graph $G:=G(A)$ calling it the {\it graph associated with $A$}, on the vertex set $\{v_1,v_2,\ldots,v_n\}$ and edge set $E(G).$ Then $\{v_i,v_j\}\in E(G)$ if and only if at least one of $a_{ij}$ or $a_{ji}$ is nonzero. 

\begin{rem}\label{diagtree}
    If $A$ is a $\Gamma$-treediagonal matrix then $D_1AD_2$ is a $\Gamma$-treediagonal matrix for all invertible diagonal matrices $D_1,D_2.$

\end{rem}

For a given tree $\Gamma$, an $n\times n$ matrix $A=(a_{ij})$ is said to satisfy the {\it $\Gamma$-treeangle property} if for every triplet of vertices $v_i, v_j, v_k \in V$ such that $v_k$ lies in the path joining $v_i$ and $v_j$, the condition $a_{ij}a_{kk}=a_{ik}a_{kj},$ holds. Note that, this also implies that $a_{ji}a_{kk}=a_{jk}a_{ki}$, for such tuples. When the tree $\Gamma$ is clear from the context, we simply refer to this as {\it treeangle property.}

Let $A\in \mnr, n>2$, be such that the diagonal elements $a_{22},a_{33},\ldots,a_{(n-1)(n-1)}$ are nonzero. We say that $A$ has the {\it triangle property} if 
    $$a_{ij}=\dfrac{a_{ik}a_{kj}}{a_{kk}},$$
for all $i<k<j$ and $i>k>j$. It is easy to see that every matrix of order $n$ satisfying the triangle property also satisfies the $P_n$-treeangle property. Conversely, when $a_{ii}\not=0$, $i=1,\ldots,n-1$ satisfies the $P_n$-treeangle property, then it also satisfies the triangle property, and hence in this case both notions are equivalent.)

Define $d_{ij}=a_{ii}a_{jj}-a_{ij}a_{ji}.$ Observe that $d_{ii}=0$ and $d_{ij}=d_{ji}$.

\begin{rem}\label{psub}\cite[Lemma 2.5]{AKS24}
   Every principal submatrix of a matrix with the triangle property also has the triangle property.
\end{rem}

\begin{rem}\label{inv}\cite[Lemma 1]{B79}
    If $A$ has the triangle property and if $d_{i(i+1)}=0$ for some $i,~1\le i \le n-1$ then $A$ is singular.
\end{rem}

\begin{thm}\cite[Theorem 4]{kl82}\label{klein1}
     Let $A=(a_{ij}) \in \mnr$ be nonsingular and satisfy the treeangle property, with $a_{ii}\neq 0$ for interior vertices $v_i \in V.$ Then $d_{ij}\not=0$ when $\{v_i,v_j\}\in E$    and $A^{-1}=(b_{ij})$ is treediagonal. The entries of $A^{-1}$ are explicitly given by
    \begin{equation}\label{invfor}
       b_{ij}=\left\{\begin{array}{cl}
       \dfrac{-a_{ij}}{d_{ij}}, & \{v_i,v_j\}\in E,\\[3ex]
       \dfrac{a_{kk}}{d_{ik}}, & i=j, ~d(v_i)=1, ~\{v_i,v_k\}\in E,\\[3ex]
      \displaystyle  \bigg(1+\sum\limits_{v_k\in N(v_i)}\dfrac{a_{ik}a_{ki}}{d_{ik}}\bigg)\frac{1}{a_{ii}}, & i=j,~ d(v_i)\ge 2,\\
       0, & otherwise.
      \end{array}\right.
      \end{equation}
 \end{thm}
 
 \begin{thm}\cite[Theorem 5]{kl82}\label{klein2}
     If $A$ is a nonsingular treediagonal matrix, then its inverse satisfies the treeangle property.
 \end{thm}    

First, we show that verbatim statements of Theorems \ref{klein1} and \ref{klein2} are false for the Moore-Penrose inverse. Let us first present a brief review of this notion.
For a given $A \in \mnr$, let $R(A), N(A)$ and $A^T$ denote the range space, the null space and the transpose of $A$, respectively. The dimension of the subspace $R(A)$ is called the {\it rank of $A$} and denoted by $\rm rk(A)$. The {\it Moore-Penrose inverse} of $A\in \mnr$ is denoted by $A^\dag$ and is defined as the unique matrix $X \in \mnr $ that satisfies the equations 
\begin{equation}\label{MP}
AXA = A,\hspace{.25cm} XAX = X,\hspace{.25cm} (AX)^T = AX\hspace{.25cm}\hbox{and}\hspace{.25cm} (XA)^T = XA.
\end{equation}
Of course, when $A$ is an invertible matrix, then $A^{\dag}=A^{-1}.$

A tool that will be used here is the notion of the full-rank factorization. Consider a matrix $A \in \mnr$ with $\rm rk(A)=r>0$. We say that $A$ has a {\it full-rank factorization} if there exist $F\in \mathbf{M}_{n\times r}(\mathbb{R}),$ $G\in \mathbf{M}_{r\times n}(\mathbb{R})$ such that $A=FG$  with $\rm rk(F)=\rm rk(G)=r$. It is well known that any nonzero matrix has a full-rank factorization.
Let $A=FG$ be a full-rank factorization of $A$. Then, $$A^{\dag}=G^{\dag}F^{\dag}=G^T(GG^T)^{-1}(F^TF)^{-1}F^T.$$ 
For a comprehensive treatment on Moore-Penrose inverses, their properties and applications, we recommend the treatise \cite{BG03}.

\begin{ex}\label{ex1}
Let $\Gamma$ be the tree  given by
$$
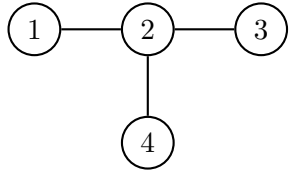
\begin{tikzpicture}[node distance={15mm}, thick, main/.style = {draw, circle}] 
\node[main] (1) {$1$}; 
\node[main] (2) [ right of=1] {$2$}; 
\node[main] (3) [ right of=2] {$3$}; 
\node[main] (4) [below of=2] {$4$};
\draw[-] (1) -- (2); 
\draw[-] (2) -- (4); 
\draw[-] (2) -- (3); 
\end{tikzpicture} $$
Then the matrix $A=\begin{pmatrix}
-2 & 2 & -4 & 2\\
-1 & 1 & -2 & 1\\
-1 & 1 & ~~1 & 1\\
-1 & 1 & -2 & 2
\end{pmatrix}$ has $\Gamma$-treeangle property. Here, 
$$A^\dag=\frac{1}{30}\begin{pmatrix}
   -8 & -4 & -10 & ~~15\\
   ~~8 & ~~4 & ~~10 & -15\\
   -4 &  -2 & ~~10 & ~~0\\
   -12 &  -6 & ~~0 & ~~30
\end{pmatrix}$$
and so $G(A^\dag)=G,$ given by the graph 
$$
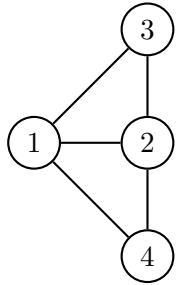
\begin{tikzpicture}[node distance={15mm}, thick, main/.style = {draw, circle}] 
\node[main] (1) {$1$}; 
\node[main] (2) [ right of=1] {$2$}; 
\node[main] (3) [ above of=2] {$3$}; 
\node[main] (4) [below of=2] {$4$};
\draw[-] (1) -- (2); 
\draw[-] (2) -- (4); 
\draw[-] (2) -- (3); 
\draw[-] (1) -- (4);
\draw[-] (3) -- (1); 
\end{tikzpicture} $$
is not a tree.
\end{ex}

\begin{ex}
The matrix $A=\begin{pmatrix}
~~1 & ~~2 & 0 & 0\\
-1 & ~~1 & 2 & 1\\
~~0 & ~~1 & 0 &0\\
~~0 & -1 & 0 & 0
\end{pmatrix}$ is $\Gamma$-treediagonal, where $\Gamma$ is the tree of Example \ref{ex1}. We have  
  $$A^\dagger=\frac{1}{10}\begin{pmatrix}
10 & 0 & -10 & ~~10\\
 0 & 0 & ~~5 & -5\\
 4 & 4 & -6 &~~6\\
 2 & 2 & -3 & ~~3
\end{pmatrix}.$$
Set $A^\dag=(b_{ij}).$ For $\{2,4\}\in E(\Gamma)$, we have $b_{22}=0$ so that $b_{22}b_{34}=0.$ However, $b_{32}b_{24}=-20$ and so $b_{22}b_{34}\neq b_{32}b_{24},$ showing that $A^\dag$ does not have the $\Gamma$-treeangle property. 
\end{ex}

These examples naturally motivate the following three questions, given that $A$ has the $\Gamma$-treeangle property for a tree $\Gamma$:

(i) How do we describe $G(A^\dag)$?\\
(ii) Under what conditions, is $G(A^\dag)$ a tree?  \\
(iii) In particular, when is $A^\dag,$ a $\Gamma$-treediagonal matrix?

Given $A \in \mathbf{M}_{n+1}(\mathbb{R}),$ we call the leading principal submatrix of $A$ corresponding to the vertex $v_i\in G(A)$ to be the matrix of order $n$ obtained by deleting the $i$-th row and the $i$-th column of $A$.
In this article, we present answers, for the particular case when a specific leading principal submatrix of $A \in \mathbf{M}_{n+1}(\mathbb{R})$ of order $n,$ is invertible. 
This is primarily motivated by the recent considerations made in \cite{KS24}.
We defer the study of the general problem, without the assumption of nonsingularity of the leading principal submatrix, and to questions relating to the converse of these statements, to a subsequent investigation.

In this context, it is worth highlighting the recent contributions \cite{AKS24, KS24}. In \cite{KS24}, the second and third authors demonstrated that, in general, the Moore-Penrose inverse (as well as another generalized inverse, viz., the group inverse) of singular matrices possessing the triangle property are not tridiagonal. Specifically, they established that when the leading principal submatrix of such a matrix is nonsingular, the corresponding generalized inverses are pentadiagonal (see \cite[Theorem 3.2]{KS24}). In a subsequent work \cite{AKS24}, a complete characterization was provided for singular matrices with the triangle property whose Moore-Penrose inverse (or group inverse) is tridiagonal. The converse implications were also thoroughly investigated.

\begin{lem}\label{nztreeangle}
    Let $A=(a_{ij})\in M_{n}(\mathbb{R}),n\ge 3$ be a matrix satisfying the $\Gamma$-treeangle property for a tree $\Gamma.$ If $a_{st}=a_{ts}=0$ for some $\{v_s,v_t\}\in E(\Gamma),$ then $A$ is a block diagonal matrix, and each block satisfies the treeangle property with respect to the corresponding subgraph of  $\Gamma.$ 
\end{lem}

\begin{proof}
    Without loss of generality, assume $s=t+1.$ Removing the edge $\{v_t,v_{t+1}\}$ from the tree $\Gamma$, partitions it into two subtrees, namely $\Gamma_1,\Gamma_2$, such that $v_t\in V(\Gamma_1)$ and $v_{t+1}\in V(\Gamma_2)$. Relabel the vertices, if necessary, so that the vertices of $\Gamma_1$ are $v_1,v_2,\ldots, v_t$, while the vertices of $\Gamma_2$ are $v_{t+1},v_{t+2},\ldots, v_n$. 
    We shall prove that the matrix $A$ has the block diagonal form $$A=\begin{pmatrix}
        A(1:t) & 0\\
        0 & A(t+1:n)
    \end{pmatrix},$$ where $A(1:t),A(t+1:n)$ satisfy the treeangle property with respect to the trees $\Gamma_1,\Gamma_2$ respectively. 
    To show this, it is enough to prove that $a_{ij}=0$ whenever $1\le i \le t$ and $j>t$, and  whenever $i>t,~1\le j \le t.$
    \\
    Consider the case $1\le i \le t$ and $j>t$. Then the vertex $v_i$
	belongs to $\Gamma_1$ and the vertex $v_j$ belongs to $\Gamma_2$. Since the unique path joining $v_i$ and $v_j$ passes through the edge $\{v_t,v_{t+1}\}$, the $\Gamma$-treeangle property gives $$a_{ij}=\frac{a_{it}a_{t(t+1)}a_{(t+1)j}}{a_{tt}a_{(t+1)(t+1)}}.$$
    Because $a_{t(t+1)}=0,$ it immediately follows that $a_{ij}=0.$ 
    Applying the same argument to the case $i>t,~1\le j \le t,$ we also obtain $a_{ij}=0.$ Therefore, $$A=\begin{pmatrix}
        A(1:t) & 0\\
        0 & A(t+1:n)
    \end{pmatrix}.$$
    Finally, since the treeangle property is preserved under principal submatrices corresponding to the connected components of the tree, the matrices $A(1:t)$ and $ A(t+1:n)$ satisfy the treeangle property with respect to the trees $\Gamma_1,\Gamma_2$.
\end{proof}
\begin{rem}\label{rnztreeangle}
    Let $A=diag(A_1,A_2,\ldots,A_k).$
    Then $$A^\dag=diag(A_1^\dag,A_2^\dag,\ldots,A_k^\dag).$$ Therefore, by the above observation and Lemma \ref{nztreeangle}, it is sufficient to restrict our attention to $\Gamma$-treeangle property matrix $A=(a_{ij})$ such that, for every edge $\{v_i,v_j\}\in E(\Gamma),$ atleast one of $a_{ij}$ or $a_{ji}$ is nonzero. 
\end{rem}

\section{Graph of Moore-Penrose Inverse}

In the first main result, we identify a set $H$ of new edges that need to be included into the edge set $E (\Gamma)$, to obtain $E(G)$, for the graph $G:=G(A^{\dag})$. Recall that, for any vertex $v$, the symbol $N(v)$ stands for all those vertices that are adjacent with $v$. We shall make use of the fact that, if nonzero columns $a^i$ and $a^j$ of a square matrix $A$  are linearly dependent, then  $N(v_i)\setminus\{v_j\}=N(v_j)\setminus \{v_i\}$ in the graph $G:=G(A).$

Suppose that $$A=\begin{pmatrix}
        B & c\\
        d^T & a_{(n+1)(n+1)}
       \end{pmatrix}$$ is a singular matrix, where $B$ be an  invertible matrix. Then the Schur complement of $B$ in $A$ is defined as, $$A/B=a_{(n+1)(n+1)}-d^TB^{-1}c.$$
       Moreover, the determinant of $A$ is given by $\det(A)=\det(B)\det(A/B).$

In what follows, we always assume that $\Gamma$ is a tree on $n+1$ vertices. Moreover,
we assume that the vertices of the tree $\Gamma$ are labeled in such a way that $v_{n+1}$ is a pendant vertex. For a given $A\in M_{n+1}(\mathbb{R})$, the leading principal submatrix corresponding to $v_{n+1}$, denoted by $A[n+1]$ (For convenience in computation, we denote it by 
$B$ in the following results), is invertible.

Our first result provides an answer to the first question posed earlier.

\begin{thm}\label{AdagG}
Let $A=(a_{ij})\in M_{n+1}(\mathbb{R}),n\ge 3$ be a singular matrix satisfying the $\Gamma$-treeangle property for a tree $\Gamma.$ Suppose that $a_{ii}\neq 0$ for all interior vertices $v_i$. Assume that $B\in M_n(\mathbb{R}),$ the leading principal submatrix of $A,$ is invertible. Let $v_{n+1}$ be a pendant vertex of $\Gamma$ and denote by $v_{r}$, the unique vertex adjacent with $v_{n+1}$. Set $G:=G(A^\dag)$. Consider the following set of edges:
$$N'_r:=\{\{v_j,v_{n+1}\}: v_j\in N(v_{r})\setminus \{v_{n+1}\} ~\text{and} ~a_{jr}\neq0\}$$
$$N''_r:=\{\{v_j,v_{n+1}\}: v_j\in N(v_{r})\setminus \{v_{n+1}\} ~\text{and} ~ a_{rj}\neq 0\}$$
and define the set of edges $H$ as follows:

$(a)$ If $a_{r(n+1)}\neq 0$ and $a_{(n+1)r}=0,$ then we set $H: = N''_r,$  \\
$(b)$ If $a_{(n+1)r}\neq 0$ and $a_{r(n+1)}=0,$ then we set $H: =N'_r.$\\
$(c)$ If both $a_{r(n+1)}$ and $a_{(n+1)r}$ are nonzero, we set $H: = N_r'\cup N_r'',$ \\
We then have the following: \\
$(1)$ $d_{ij}\neq 0$  when $\{v_i,v_j\}\in E\big (\Gamma\setminus{v_{n+1}}\big )$. \\ 
$(2)$ Let  
$$s_{r}=\frac{1}{a_{rr}}\Big(1+\displaystyle \sum_{\{k|v_k\in N(v_{r})\}\atop  k\neq n+1}\frac{a_{rk}a_{kr}}{d_{rk}}\Big).$$
$(i)$ If $s_{r}\neq0,$ then $E(G)=E(\Gamma)\cup H;$
\\
$(ii)$ If $s_{r}=0,$ then $E(G)=(E(\Gamma)\setminus\{v_{r},v_{n+1}\}) \cup H.$
\end{thm}
\begin{proof}
We have
\begin{eqnarray}\label{Arep}
A=\begin{pmatrix}
        B & c\\
        d^T & a_{(n+1)(n+1)}
       \end{pmatrix}.
\end{eqnarray} 
Since $A$ has $\Gamma$-treeangle property, $v_{n+1}$ is a pendant vertex adjacent to $v_{r}$, and the unique path from any vertex to the vertex $v_{n+1}$ passes through $v_r$, we obtain
\begin{eqnarray}\label{eqc}
    c& = & (a_{1(n+1)},a_{2(n+1)},\ldots,a_{n(n+1)})^T\nonumber\\& = &
    \Big(\frac{a_{1r}a_{r(n+1)}}{a_{rr}},\frac{a_{2r}a_{r(n+1)}}{a_{rr}},\ldots,\frac{a_{nr}a_{r(n+1)}}{a_{rr}}\Big)^T\nonumber\\ & = &
    \frac{a_{r(n+1)}}{a_{rr}}(a_{1r},a_{2r},\ldots,a_{nr})^T\nonumber\\
    & = & \frac{a_{r(n+1)}}{a_{rr}} Be^{r},
\end{eqnarray} 

We also obtain, 
\begin{eqnarray}\label{eqd}
d^T & = & (a_{(n+1)1},a_{(n+1)2},\ldots,a_{(n+1)n})\nonumber\\& = &
    \Big(\frac{a_{(n+1)r}a_{r1}}{a_{rr}},\frac{a_{(n+1)r}a_{r2}}{a_{rr}},\ldots,\frac{a_{(n+1)r}a_{rn}}{a_{rr}}\Big)\nonumber\\ & = &
    \frac{a_{(n+1)r}}{a_{rr}}(a_{r1},a_{r2},\ldots,a_{rn})\nonumber\\
    & = & \frac{a_{(n+1)r}}{a_{rr}} (e^{r})^TB.
\end{eqnarray}
The matrix $A$ is singular and $B$ is nonsingular. So, the Schur complement of $B$ in $A$ is zero. Thus, we have 
\begin{eqnarray*}
a_{(n+1)(n+1)} & = & d^TB^{-1}c
\nonumber\\
& = &  \frac{a_{r(n+1)}a_{(n+1)r}}{a_{rr}^2}(e^{r})^TBe^{r}\\
& = & \frac{a_{r(n+1)}a_{(n+1)r}}{a_{rr}^2}b_{rr}\\
& = & \frac{a_{r(n+1)}a_{(n+1)r}}{a_{rr}},
\end{eqnarray*}
so that 
\begin{eqnarray}\label{schureq1}
a_{rr}a_{(n+1)(n+1)} =  a_{r(n+1)}a_{(n+1)r}.
\end{eqnarray}
Thus, we have the block decomposition of $A$ given as
$$A=\begin{pmatrix}
    B & \frac{a_{r(n+1)}}{a_{rr}} B{\bf e}^r\\
    \frac{a_{(n+1)r}}{a_{rr}} ({\bf e}^r)^TB & \frac{a_{r(n+1)}a_{(n+1)r}}{a_{rr}} 
\end{pmatrix},$$
If we set 
$$F=  \begin{pmatrix}
       B\\\frac{a_{(n+1)r}}{a_{rr}}({\bf e}^r)^TB 
       \end{pmatrix}=\begin{pmatrix}
       I\\\frac{a_{(n+1)r}}{a_{rr}}({\bf e}^r)^T
       \end{pmatrix}B,$$\\
       and $$G=
       \begin{pmatrix}
           I & \frac{a_{r(n+1)}}{a_{rr}}{ {\bf e}^r}
       \end{pmatrix},$$
       then, $A=FG$ is a full rank factorization. 
We have
\begin{eqnarray*}
F^TF & = & B^T{\begin{pmatrix}
          I, \frac{a_{(n+1)r}}{a_{rr}} {\bf e}^r
      \end{pmatrix}}{\begin{pmatrix}
          I\\\frac{a_{(n+1)r}}{a_{rr}} ({\bf e}^r)^T
      \end{pmatrix}}B\\
      & = & B^T(I+\frac{a_{(n+1)r}^2}{a_{rr}^2}{\bf e}^r({\bf e}^r)^T)B\\
      & = & B^TD_1B.
      \end{eqnarray*}
where $D_1=I+\frac{a_{(n+1)r}^2}{a_{rr}^2}{\bf e}^r({\bf e}^r)^T$ is an invertible diagonal matrix. Since $A=FG$ is a full rank factorization, $F^TF $ is invertible and one has $$(F^TF)^{-1}=B^{-1}D_1^{-1}(B^T)^{-1}.$$
Now, \begin{eqnarray*}
    F^{\dagger}=(F^TF)^{-1}F^T & = & B^{-1}D_1^{-1}(B^T)^{-1}B^T{\begin{pmatrix}
     I, \frac{a_{(n+1)r}}{a_{rr}} {\bf e}^r
    \end{pmatrix}}\\
    & = & B^{-1}D_1^{-1}{\begin{pmatrix}
     I,\frac{a_{(n+1)r}}{a_{rr}} {\bf e}^r
    \end{pmatrix}}.
    \end{eqnarray*}
Also,
      \begin{eqnarray*}
        GG^T & = &
      {\begin{pmatrix}
          I, \frac{a_{r(n+1)}}{a_{rr}}{\bf e}^r
      \end{pmatrix}}{\begin{pmatrix}
          I\\
    \frac{a_{r(n+1)}}{a_{rr}} ({\bf e}^r)^T
      \end{pmatrix}}\\
      & = & I+\frac{a_{r(n+1)}^2}{a_{rr}^2}{\bf e}^r({\bf e}^r)^T.
       \end{eqnarray*}
Set $D_2:=I+\frac{a_{r(n+1)}^2}{a_{rr}^2}{\bf e}^r({\bf e}^r)^T$. Then $D_2$ is also an invertible diagonal matrix. Thus, 
  \begin{eqnarray*}
    G^{\dagger} & = & G^T(GG^T)^{-1}=
    \begin{pmatrix}
     I\\
    \frac{a_{r(n+1)}}{a_{rr}} ({\bf e}^r)^T
    \end{pmatrix}D_2^{-1}.
    \end{eqnarray*}
Thus, the Moore-Penrose inverse of $A$ is given by  
\begin{eqnarray}\label{Adagfor}
    A^{\dagger}& = & G^{\dagger}F^{\dagger} \nonumber\\
    & = & {\begin{pmatrix}
    I\\
    \frac{a_{r(n+1)}}{a_{rr}} ({\bf e}^r)^T
\end{pmatrix}}D_2^{-1}B^{-1}D_1^{-1}{\begin{pmatrix}
    I,\frac{a_{(n+1)r}}{a_{rr}}{\bf e}^r
\end{pmatrix}} \nonumber\\
& = &\begin{pmatrix}
    D_2^{-1}B^{-1}D_1^{-1}& \frac{a_{(n+1)r}}{a_{rr}} D_2^{-1}B^{-1}D_1^{-1}{\bf e}^r\\
    \frac{a_{r(n+1)}}{a_{rr}} ({\bf e}^r)^TD_2^{-1}B^{-1}D_1^{-1}& \frac{a_{(n+1)r}a_{r(n+1)}}{a_{rr}^2}  ({\bf e}^r)^TD_2^{-1}B^{-1}D_1^{-1}{\bf e}^r\\
\end{pmatrix}
 \end{eqnarray}
 Note that $\Gamma \setminus \{v_{n+1}\}$ is a tree and that $B$ has the $\Gamma \setminus \{v_{n+1}\}$-treeangle property. Since $B$ is invertible, by Theorem \ref{klein1}, we infer that $d_{ij}\neq 0$ for $i,j<n+1,$ and $B^{-1}$ is $\Gamma \setminus \{v_{n+1}\}$-treediagonal.  Since $D_1,D_2$ are invertible diagonal matrices, using Remark \ref{diagtree}, $D_2^{-1}B^{-1}D_1^{-1}$ is $\Gamma \setminus \{v_{n+1}\}$-treediagonal.  Observe that the tree $\Gamma \setminus \{v_{n+1}\}$ is the same as $G \setminus \{v_{n+1}\}$. Set $A^{\dag}=:(c_{ij})$,$~D_1=:diag(d_1,d_2,\ldots,d_n)$ and $D_2=:diag(q_1,q_2,\ldots,q_n).$ Then $\{v_r,v_{n+1}\}\in E(G)$ if and only if either  $c_{r(n+1)}$ or $c_{(n+1)r}$ is nonzero.  From (\ref{invfor}) and (\ref{Adagfor}), we obtain  $$c_{rr}=({\bf e}^r)^TD_2^{-1}B^{-1}D_1^{-1}{\bf e}^r=\frac{1}{a_{rr}d_rq_r}\Big(1+\displaystyle \sum_{\{k|v_k\in N(v_{r})\}\atop  k\neq n+1}\frac{a_{rk}a_{kr}}{d_{rk}}\Big)=\frac{s_{r}}{d_rq_r}.$$ 
Thus,\begin{eqnarray}\label{eqsi}
    c_{(n+1)r} & = & \nonumber\frac{a_{(n+1)r}}{a_{rr}}({\bf e}^r)^TD_2^{-1}B^{-1}D_1^{-1}{\bf e}^r \\\nonumber
    & = & \frac{a_{(n+1)r}}{a_{rr}}c_{rr} \\
    & = & \frac{a_{(n+1)r}s_r}{a_{rr}d_rq_r}.
\end{eqnarray}
Similarly, \begin{eqnarray}\label{eqsi1}
    c_{r(n+1)}&=&\frac{a_{r(n+1)}}{a_{rr}}({\bf e}^r)^TD_2^{-1}B^{-1}D_1^{-1}{\bf e}^r\nonumber\\&=&\frac{a_{r(n+1)}s_r}{a_{rr}d_rq_r}.
\end{eqnarray} 
It now follows that, $\{v_r,v_{n+1}\}\in E(\Gamma)$, if and only if either $a_{r(n+1)}$ or $a_{(n+1)r}$ is nonzero. Equivalently, $\{v_r,v_{n+1}\}\in E(G)$ if and only if $s_r\neq 0.$ So far, we have proved the following:
If $H$ denotes the set of edges that are incident on $v_{n+1}$, excluding the edge $\{v_r,v_{n+1}\}$, then
\begin{enumerate}
    \item $E(G)=E(\Gamma)\cup H,$ if $s_r\neq0.$
    \item $E(G)=(E(\Gamma)\setminus\{v_r,v_{n+1}\}) \cup H,$ if $s_r=0$.
\end{enumerate} 

From (\ref{Adagfor}), we observe that the last column of $A^{\dag}$ is a constant multiple of the $r$-th column of $A^{\dag}$, where the constant is given by $\frac{a_{r(n+1)}}{a_{rr}}$. Similarly, the last row of $A^{\dag}$ is a constant multiple of the $r$-th row of $A^{\dag}$, where the multiplicative constant in this case, is given by $\frac{a_{(n+1)r}}{a_{rr}}$.

Thus, if $a_{r(n+1)}$ and $a_{(n+1)r}$ are both nonzero, then 
the vertices $v_r$ and $v_{n+1}$ have the same neighbours in $G.$ Thus, $(a)$ is established.\\
Suppose $a_{r(n+1)}\neq0 $ and $a_{(n+1)r}=0.$ 
Using (\ref{Adagfor}), we conclude that the neighbours of $v_{n+1}$ in $G$ only depend on the nonzero entries in the $r$-th column. Since $B$ is invertible,  using (\ref{invfor}), we obtain $c_{jr}=\frac{-a_{jr}}{d_jq_rd_{rj}}$ for $\{v_j,v_r\}\in E(\Gamma)$ and $j\neq n+1.$ Thus, $\{v_j,v_{n+1}\}\in H$ if and only if $a_{jr}\neq 0$ and $v_j\in N(v_r)$, proving $(b)$. \\
 Since the treeangle property is invariant under the operation of matrix transposition, $(c)$ follows similarly.
\end{proof}

\begin{rem}
In the notation of the above result, let $v_r$ be the unique vertex adjacent to $v_{n+1}.$ Then, by (\ref{eqc}),(\ref{schureq1}), it follows that the $r$th column and $(n+1)$th column of $A$ are linearly dependent. Consider a vertex $v_i,~i\neq r,n+1.$ Then the principal submatrix $A[i]$ is such that its columns corresponding to the vertices $v_r, v_{n+1}$ are linearly dependent. Hence $A[i]$ is singular. Thus, $v_r$ is the unique vertex satisfying the said property. Next, let $A$ be a matrix for which this property holds. Consider a permutation of the rows and columns of $A$ in such a way that the last row/column corresponds to another pendant vertex of $~\Gamma.$ Denote the new matrix obtained thus, by $C:=PAP^T$. By the argument above, the principal submatrix obtained by deleting the last row and the last column of $C$, is singular. However, the graph of $C^{\dag}$ is the same as the graph of $A^{\dag}$, but for a relabelling of the vertices (See Remark \ref{lps}, for instance).
\end{rem}

\begin{ex}\label{ex12}
Let $\Gamma$ be the tree given by
$$
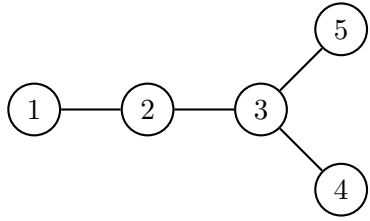
\begin{tikzpicture}[node distance={15mm}, thick, main/.style = {draw, circle}] 
\node[main] (1) {$1$}; 
\node[main] (2) [ right of=1] {$2$}; 
\node[main] (3) [right of=2] {$3$}; 
\node[main] (4) [below right of=3] {$4$}; 
\node[main] (5) [above right of=3] {$5$}; 
\draw[-](1) -- (2); 
\draw[-] (2) -- (3); 
\draw[-] (3) -- (4); 
\draw[-] (3) -- (5);
\end{tikzpicture}$$
It may be verified that both the matrices \begin{center}
 $A_1=\begin{pmatrix}
~~1 & ~~1 & -1 & ~~1 & ~~1\\
-2 & -1 & ~~1 & -1 & -1\\
~~2 & ~~1 & ~~1 & -1 & -1\\
-2 & -1 & -1 & -1 & ~~1\\
-2 & -1 & -1 & ~~1 & ~~1
\end{pmatrix}$ and $A_2=\begin{pmatrix}
~~1 & ~~1 & -1 & ~~2 & ~~1\\
-2 & -1 & ~~1 & -2 & -1\\
~~2 & ~~1 & ~~1 & -2 & -1\\
-2 & -1 & -1 & -1 & ~~1\\
-2 & -1 & -1 & ~~2 & ~~1
\end{pmatrix}$
\end{center}
have the $\Gamma$-treeangle property. 
Then
we have, 

\begin{center}
$A_2^\dag=\frac{1}{24}\begin{pmatrix}
-24 & -24 & ~~0 & ~~0 & ~~0\\
~~48 & ~~36 & ~~6 & ~~0 & -6 \\
~~0 & ~~6 & -1 & -8 & ~~1\\
~~0 & ~~0 & -4 & -8 & ~~4\\
~~0 & -6 & ~~1 & ~~8 & -1
\end{pmatrix}.$ 
\end{center}
We consider the case when $v_5$ is taken as the pendant vertex, since the leading principal submatrix corresponding to $v_5$ is nonsingular. Here, $r=3, ~s_{3}\neq0$  and $a_{35},a_{53}$ are both nonzero. From $(c)$ of Theorem \ref{AdagG},  we infer that $$H=\{\{v_j,v_{5}\}: v_j\in N(v_{3})\setminus \{v_{5}\}\}=\{\{2,5\},\{4,5\}\}.$$ Thus, we obtain the edge set, $E(G)=E(\Gamma)\cup(\{2,5\},\{4,5\}).$ The graph $G$ given below
$$
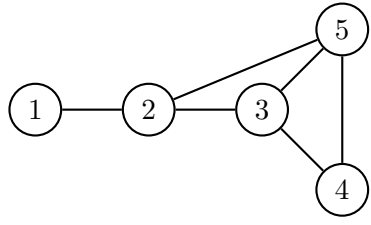
\begin{tikzpicture}[node distance={15mm}, thick, main/.style = {draw, circle}] 
\node[main] (1) {$1$}; 
\node[main] (2) [ right of=1] {$2$}; 
\node[main] (3) [right of=2] {$3$}; 
\node[main] (4) [below right of=3] {$4$}; 
\node[main] (5) [above right of=3] {$5$}; 
\draw[-](1) -- (2); 
\draw[-] (2) -- (3); 
\draw[-] (3) -- (4); 
\draw[-] (3) -- (5);
\draw[-] (2) -- (5);\draw[-] (4) -- (5);
\end{tikzpicture}$$  
confirms that $G=G(A_2^\dag).$

For the matrix $A_1,$ 
\begin{center}
$A_1^\dag=\frac{1}{4}\begin{pmatrix}
    -4 & -4 & ~~0 & ~~0 & ~~0\\
    ~~8 & ~~6 & ~~1 & ~~0 & -1\\
    ~~0 & ~~1 & ~~0 & -1 & ~~0\\
    ~~0 & ~~0 & -1 & -2 & ~~1\\
    ~~0 & -1 & ~~0 & ~~1 & ~~0
\end{pmatrix}$
\end{center}
For  $A_1$, we have $r=3, ~s_{3}=0$  and $a_{35},a_{53}$ are both nonzero. From $(c)$ of Theorem \ref{AdagG}, using similar argument as above, we have $H=\{\{2,5\},\{4,5\}\}$. Thus, $E(G)=(E(\Gamma)\setminus\{3,5\})\cup(\{2,5\},\{4,5\}).$ Thus, the graph $G$ is:
$$
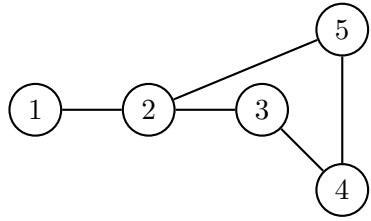
\begin{tikzpicture}[node distance={15mm}, thick, main/.style = {draw, circle}] 
\node[main] (1) {$1$}; 
\node[main] (2) [ right of=1] {$2$}; 
\node[main] (3) [right of=2] {$3$}; 
\node[main] (4) [below right of=3] {$4$}; 
\node[main] (5) [above right of=3] {$5$}; 
\draw[-](1) -- (2); 
\draw[-] (2) -- (3); 
\draw[-] (3) -- (4); 
\draw[-] (2) -- (5);\draw[-] (4) -- (5);
\end{tikzpicture}$$  
It is verified that $G=G(A_1^\dag)$. 
\end{ex}

\begin{rem}\label{lps}
There exist three pendant vertices in the graph of Example \ref{ex12}, viz., $v_1,v_4$ and $v_5$. The corresponding leading principal submatrices of $A_1$ (assigned subscripts, according to the label of the pendant vertex) are respectively,
$$B_1=\begin{pmatrix}
-1 & ~~1 & -1 & -1\\ ~~1 & ~~1 & -1 & -1\\ -1 & -1 & -1 & ~~1\\
-1 & -1 & ~~1 & ~~1 \end{pmatrix},\hspace{.15cm}B_4=\begin{pmatrix}
~~1 & ~~1 & -1 & ~~1\\ -2 & -1 & ~~1 & -1\\ ~~2 & ~~1 & ~~1 & -1\\ -2 & -1 & -1 & ~~1 \end{pmatrix},\hspace{.15cm}
B_5=\begin{pmatrix}
~~1 & ~~1 & -1 & ~~1\\ -2 & -1 & ~~1 & -1\\
~~2 & ~~1 & ~~1 & -1\\
-2 & -1 & -1 & -1 \end{pmatrix}.$$
Since ${\rm det}(B_1)={\rm det}(B_4)=0$ and ${\rm det}(B_5)=-4,$ Theorem 2.1 is applicable only for the vertex $v_5$. 
\end{rem}

\begin{ex}\label{ex122}
    Let $\Gamma$ be the tree given by
$$
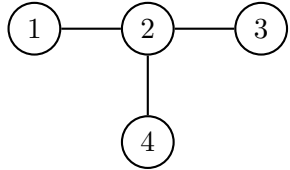
\begin{tikzpicture}[node distance={15mm}, thick, main/.style = {draw, circle}] 
\node[main] (1) {$1$}; 
\node[main] (2) [ right of=1] {$2$}; 
\node[main] (3) [right of=2] {$3$}; 
\node[main] (4) [below  of=2] {$4$}; 
\draw[-](1) -- (2); 
\draw[-] (2) -- (3); 
\draw[-] (2) -- (4); 
\end{tikzpicture}$$
Both the matrices 
\begin{center}
 $A_1=\begin{pmatrix}
0 & -3 & -21 & ~~3\\
5 & ~~1 & ~~7 & -1\\
0 & ~~0 & -1 & ~~0\\
0 & ~~0& ~~0 &~~0
\end{pmatrix}$ and $A_2=\begin{pmatrix}
1 & 0 & ~~0 & ~~0\\
2 & 1 & -1 & -1\\
2 & 1 & ~~5 & -1\\
0 & 0 & ~~0 &~~0
\end{pmatrix}$
\end{center}
have the $\Gamma$-treeangle property. Observe that the tree has three pendant vertices, $1$, $3$, and $4$, but in both cases, $A_1$ and $A_2$, the unique pendant vertex whose leading principal matrix is nonsingular is $4$. For the matrix $A_1$, $r=2~s_{2}=0, ~a_{24}\neq 0$ and $a_{42}= 0$. Thus, from $(a)$ of Theorem \ref{AdagG}, $$H=\{\{v_j,4\}: v_j\in N(2)\setminus \{4\} ~\text{and} ~a_{2j}\neq0\}.$$ Since $a_{21}\not=0$ and $a_{23}\not=0$,  $H=\{\{1,4\},\{3,4\}\}.$ Thus, $E(G)=(E(\Gamma)\backslash{\{2,4\}})\cup\{\{1,4\},\{3,4\}\}.$
We have, 
\begin{center}
$A_1^\dag=\frac{1}{30}\begin{pmatrix}
    ~~2 & 6 & ~~0 & 0\\
    -5 & 0 & ~~105 & 0\\
    ~~0 & 0 & -30 & 0\\
    ~~5 & 0 & -105 & 0
\end{pmatrix}$
\end{center}
and $G(A_1^\dag)$ is given by $$
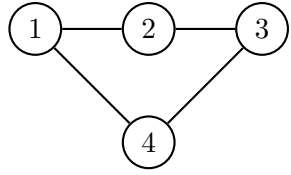
\begin{tikzpicture}[node distance={15mm}, thick, main/.style = {draw, circle}] 
\node[main] (1) {$1$}; 
\node[main] (2) [ right of=1] {$2$}; 
\node[main] (3) [right of=2] {$3$}; 
\node[main] (4) [below  of=2] {$4$}; 
\draw[-](1) -- (2); 
\draw[-] (2) -- (3); 
\draw[-] (3) -- (4); 
\draw[-] (4) -- (1);
\end{tikzpicture}$$  
  
For the matrix $A_2$, $r=2~s_{2}\neq 0$, $a_{24}\neq 0$ and $a_{42}=0$. Even though $H=\{\{1,4\},\{3,4\}\}$ is the same as earlier, $E(G)=E(\Gamma)\cup\{\{1,4\},\{3,4\}\}.$ Therefore, we have the following graph: 
 $$
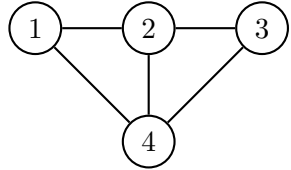
\begin{tikzpicture}[node distance={15mm}, thick, main/.style = {draw, circle}] 
\node[main] (1) {$1$}; 
\node[main] (2) [ right of=1] {$2$}; 
\node[main] (3) [right of=2] {$3$}; 
\node[main] (4) [below  of=2] {$4$}; 
\draw[-] (1) -- (2); 
\draw[-] (2) -- (3);
\draw[-] (2) -- (4);
\draw[-] (3) -- (4); 
\draw[-] (4) -- (1);
\end{tikzpicture}$$
confirming that $G=G(A_2^{\dag}),$ since $A_2^\dag$ is given by
\begin{center}
$A_2^\dag=\frac{1}{12}\begin{pmatrix}
~~12 & ~~0 & ~~0 & 0 \\
-12 & ~~5 & ~~1 & 0 \\
~~0 & -2 & ~~2 & 0 \\
~~12 & -5 & -1 & 0
\end{pmatrix}.$  
\end{center}
\end{ex}

Theorem \ref{AdagG} is motivated by the following result, which in turn, is obtained as a consequence.

\begin{thm}\label{mppenta}\cite[Theorem 3.2]{KS24}
  Let $A=(a_{ij})\in M_{n+1}(\mathbb{R})$ be a singular matrix with the triangle property such that the leading principal submatrix of $A$ has rank $n$. Then $A^{\dag}$ is a pentadiagonal matrix.   \end{thm}
\begin{proof}
    Let $A\in M_{n+1}(\mathbb{R})$ be a singular matrix with the triangle property. Then, by definition, $a_{22}, \ldots , a_{(n-1)(n-1)}, a_{nn}$ are nonzero. Observe that $A$ may be viewed as having the $P_{n+1}$-treeangle property and so the leading principal submatrix of order $n$, of $A^\dag$, is tridiagonal. Further, by Remark \ref{psub}, $B$ is an invertible matrix that possesses the triangle property, so that, by Remark \ref{inv}, $d_{n(n-1)} \neq 0.$ Let $s_n$ be defined as in Theorem \ref{AdagG}. Then 
 \begin{center}
  $s_n=\frac{1}{a_{nn}}\big(1+\frac{a_{(n-1)n}a_{n(n-1)}}{d_{n(n-1)}}\big)=\frac{a_{(n-1)(n-1)}}{d_{n(n-1)}}\neq 0$.    
 \end{center}
 By Theorem \ref{AdagG}, $E(G)=E(P_{n+1})\cup H,$ where $G:=G(A^{\dag}).$  Since $N(v_n)=\{v_{n-1},v_{n+1}\},$ in all the three cases $(a),(b),(c)$ of Theorem \ref{AdagG}, we obtain $H \subseteq\{\{v_{n-1},v_{n+1}\}\}.$ Hence, the first $n-2$ entries of the last column of $A^{\dag}$, as well as the first $n-2$ entries of its last row, are zero. Thus, $A^\dag$ is a pentadiagonal matrix. \end{proof}

\begin{rem}
It may happen that $A^{\dag}$ is even tridiagonal (so that it is trivially pentadiagonal), as the following example shows. From the proof of the result above, it is clear that this happens precisely when $H$ is empty.
Let $$A=\begin{pmatrix}
    -1 & 2 & ~~2 & -4 & -4\\
    ~~1 & 1 & ~~1 & -2 & -2\\
    ~~1 & 1 & -1 & ~~2 & ~~2\\
    ~~0 & 0 & ~~0 & ~~1 & ~~1\\
    ~~0 & 0 & ~~0 & ~~0 & ~~0
\end{pmatrix}.$$ Then $A$ is a singular matrix with the triangle property, and $$A^\dag=\frac{1}{6}\begin{pmatrix}
    -2 & ~~4 & ~~0 & 0 & 0\\
    ~~2 & -1 & ~~3 & 0 & 0\\
    ~~0 & ~~3 & -3 & 12 & 0\\
    ~~0 & ~~0 & ~~0 & 3 & 0\\
    ~~0 & ~~0 & ~~0 & 3 & 0
\end{pmatrix},$$ 
is a tridiagonal matrix. Here, $s_4=\frac{1}{2}$, $a_{54}=0,a_{45}=1\neq 0.$ Then $$H=N_r''=\{\{v_j,5\}:v_j\in N(4)\backslash{\{5\}}, ~a_{4j}\neq 0\}=\phi.$$
\end{rem}

 For a finite set $X$, we let $|X|$ denote its cardinality. In the next result, we identify precisely when $G(A^\dag)$ is a tree, thereby answering the second question in the Introduction.

 \begin{cor}\label{mptree}
  Let $\Gamma$ be a tree and $A=(a_{ij})\in M_{n+1}(\mathbb{R}),n\ge 3$ be a singular matrix satisfying the $\Gamma$-treeangle property for a tree $\Gamma.$ Suppose that $a_{ii}\neq 0$ for all interior vertices $v_i$. Assume that $B\in M_n(\mathbb{R}),$ the leading principal submatrix of $A,$ is invertible. Let $v_{r}$ be the unique vertex incident with $v_{n+1}$. If $G:=G(A^\dag),$ then $G$ is a tree iff any one of the following holds: \begin{enumerate}
      \item If $s_r\neq 0,$ then $H$ is empty.
\item If $s_r=0,$ then $H$ is singleton. 
  \end{enumerate}
Here, $H$ is as defined in Theorem \ref{AdagG}.

  \end{cor}
  \begin{proof}
    Assume that $G$ is a tree. Let $s_{r}\neq0$, so that, $E(G)=E(\Gamma)\cup H$, by Theorem \ref{AdagG}. Since $E(\Gamma)\cap H=\phi$, we then have $|E(G)|=|E(\Gamma)|+|H|.$
   Since $G,\Gamma$ are trees on $n+1$ vertices, $|E(G)|=|E(\Gamma)|=n$ and so $H$ is empty.\\
   If, on the other hand, $s_{r}=0$ then, again by Theorem \ref{AdagG}, $E(G)=\{E(\Gamma)\backslash{\{v_r,v_{n+1}\}}\}\cup H$.
    As before, $|E(G)|=|E(\Gamma)|-1+|H|$, since  $\{E(\Gamma)\backslash{\{v_r,v_{n+1}\}}\}\cap H=\phi$. As $G,\Gamma$ are trees on $n+1$ vertices, we have that $H$ is a singleton. 
      
Conversely, suppose that 1 holds. By $(2.i)$ of Theorem \ref{AdagG}, $E(G)=E(\Gamma)$ and so one has $G=\Gamma$. If, on the other hand, 2 holds, then using $(2.ii)$ of Theorem \ref{AdagG}, we obtain $|E(G)|=n$. Now, if $G$ is connected, then we can conclude that $G$ is a tree. Suppose that $G$ is not connected. Then $A^\dag$ is a block diagonal matrix. As was argued in the proof of Theorem \ref{AdagG}, after the expression for $A^{\dag}$ in (\ref{Adagfor}),  the matrix $A^\dag[n+1]$ is $\Gamma\setminus \{v_{n+1}\}$ -treediagonal. Thus $A^\dag[n+1]$ is not a block diagonal matrix. This means that$$A^\dag=\begin{pmatrix}
    D_2^{-1}B^{-1}D_1^{-1} & 0\\
    0 & \lambda
\end{pmatrix}.$$ 
However, since $A^\dag$ is singular, $\lambda=0$ and so 
$$A=\begin{pmatrix}
    D_1BD_2 & 0 \\
    0 & 0
    \end{pmatrix}.$$ 
Thus, $a_{r(n+1)}=a_{(n+1)r}=0$, However, by Remark \ref{rnztreeangle}, this is impossible because $v_r$ is adjacent to $v_{n+1}.$
\end{proof}

\begin{ex} The matrix   
$A=\begin{pmatrix}
        0 & 4 & 4 & -4 & 0\\
        1 & 2 & 2 & -2 & 0\\
        0 & 0 & 2 & ~~0 & 0\\
        0 & 0 & 0 & ~~2 & 0\\
        1 & 2 & 2 & -2 & 0
\end{pmatrix}$ possesses the $\Gamma$-treeangle property, where
 $\Gamma$ is:
$$
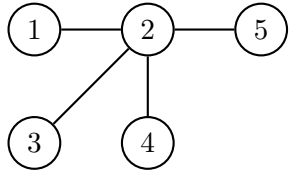
\begin{tikzpicture}[node distance={15mm}, thick, main/.style = {draw, circle}] 
\node[main] (1) {$1$}; 
\node[main] (2) [right of=1] {$2$}; 
\node[main] (3) [below of=1] {$3$}; 
\node[main] (4) [below of=2] {$4$};
\node[main] (5) [right of=2] {$5$};
\draw[-] (1) -- (2); 
\draw[-] (2) -- (4); 
\draw[-] (2) -- (3); 
\draw[-] (5) -- (2); 
\end{tikzpicture} $$ 
Thus, $r=2, ~d_{12}=-4,~d_{23}=4,~d_{24}=4$ and so $$s_2=\frac{1}{a_{22}}\big(1+\frac{a_{21}a_{12}}{d_{12}}+\frac{a_{23}a_{32}}{d_{23}}+\frac{a_{24}a_{42}}{d_{24}}\big)=0.$$Here, $$H=\{\{v_j,5\}:v_j\in N(2)\backslash{\{5\}},~a_{j2}\neq 0\}=\{\{1,5\}\}$$ and so item 2 in the above corollary is satisfied. Note that, if $\Gamma_1$ is the tree given by
$$
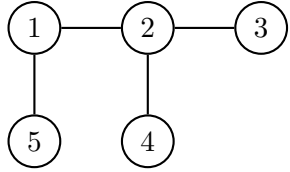
\begin{tikzpicture}[node distance={15mm}, thick, main/.style = {draw, circle}] 
\node[main] (1) {$1$}; 
\node[main] (2) [right of=1] {$2$}; 
\node[main] (3) [right of=2] {$3$}; 
\node[main] (4) [below of=2] {$4$};
\node[main] (5) [below of=1] {$5$};
\draw[-] (1) -- (2); 
\draw[-] (2) -- (4); 
\draw[-] (2) -- (3); 
\draw[-] (5) -- (1); 
\end{tikzpicture} $$ 
then $$A^\dag=\frac{1}{4}\begin{pmatrix}
        -2 & 2 & ~~0 & 0 & 2\\
        ~~1 & 0 & -2 & 2 & 0\\
        ~~0 & 0 & ~~2 & 0 & 0\\
        ~~0 & 0 & ~~0 & 2 & 0\\
        ~~0 & 0 & ~~0 & 0 & 0
\end{pmatrix}$$ may be verified to be $\Gamma_1$-treediagonal. This confirms the conclusion of the result. 
\end{ex}

To conclude, we answer the third question of the Introduction, by characterizing when the graph of the Moore-Penrose inverse of $A$ has the treediagonal property relative to the tree, with respect to which $A$ has the treeangle property.
 
 \begin{thm}\label{dag1}
 Let $\Gamma$ be a tree and $A\in M_{n+1}(\mathbb{R}),n\ge 3$ be a singular matrix possessing the $\Gamma$-treeangle property. Assume further, that $a_{ii}\neq 0$ for interior vertices $v_i$. Assume that $B\in M_n(\mathbb{R}),$ the leading principal submatrix of $A,$ is invertible. Let $v_{n+1}$ be a pendant vertex of $\Gamma$ and denote by $v_r,$ the unique vertex adjacent to $v_{n+1}.$ Then $A^{\dag}$ is $\Gamma$-treediagonal if and only if $A$ or its transpose is of the form:
    $$\begin{pmatrix}
        B & 0 \\
        \zeta ({\bf e}^r)^TB & 0
    \end{pmatrix},$$ for some $\zeta\neq 0$, with the $r$th column of $B$ being a nonzero multiple of ${\bf e}^r.$
\end{thm}
  \begin{proof} Assume that $A^\dag$ is $\Gamma$-treediagonal. 
Since $A^\dag$ is $\Gamma$-treediagonal, using (\ref{Adagfor}), we can conclude that either \begin{eqnarray}\label{zeta11}
    \frac{a_{(n+1)r}}{a_{rr}}D_2^{-1}B^{-1}D_1^{-1}{\bf e}^r=\zeta_1 {\bf e}^r
\end{eqnarray}
       or
\begin{eqnarray}\label{2zeta11}
    \frac{a_{r(n+1)}}{a_{rr}}({\bf e}^r)^TD_2^{-1}B^{-1}D_1^{-1}=\zeta_2 ({\bf e}^r)^T
\end{eqnarray}
for $\zeta_1,\zeta_2\in \mathbb{R}.$
Suppose that $\zeta_1\neq 0.$ Then $a_{(n+1)r}\neq 0.$ We consider $(\ref{zeta11})$. Since $D_1,D_2$ are invertible diagonal matrices, we obtain $B^{-1}{\bf e}^r=\delta_1 {\bf e}^r$ for some $\delta_1\neq 0$. Thus, \begin{eqnarray}\label{beteq1}
       B{\bf e}^r=\frac{1}{\delta_1}{\bf e}^r.
   \end{eqnarray}
By (\ref{eqd}), we obtain \begin{eqnarray*}
       d^T&=&\frac{a_{(n+1)r}}{a_{rr}}({\bf e}^r)^TB\nonumber
   \end{eqnarray*} 
   {\bf Claim}: $a_{r(n+1)}=0.$\\
   Suppose not. Then, both $a_{r(n+1)},a_{(n+1)r}$ are nonzero. Note that, by the fact that $A^{\dag}$ is $\Gamma$-treediagonal, $G=\Gamma.$ By part (c) of Theorem \ref{AdagG}, we have $H=N'_r\cup N''_r.$ Note that  $H$ is the set of all edges incident with the vertex $v_{n+1}$ in $G$. Since $n\ge 3$, the vertex $v_r$ has atleast one neighbor, say $v_{\ell}.$ By Remark \ref{rnztreeangle}, atleast one of $a_{\ell r}$ or $a_{r\ell}$ is nonzero. Therefore, by the definitions of $N_r'$ and $N_r''$, it follows that the vertex $v_{\ell}$ is also adjacent to $v_{n+1}$, which is a contradiction to the fact $v_{n+1}$ is pendant vertex of $G.$

   Thus, $$A=\begin{pmatrix}
        B & 0 \\
        \zeta ({\bf e}^r)^TB & 0
    \end{pmatrix}$$ and satisfies (\ref{beteq1}).\\ Suppose (\ref{2zeta11})  holds. Since the treeangle property is invariant under the operation of matrix transposition, $A^T$ takes the prescribed form. \\
   Conversely, suppose that $A=\begin{pmatrix}
        B & 0 \\
        \zeta ({\bf e}^r)^TB & 0
    \end{pmatrix}$ for $\zeta \neq 0$,~ $B{\bf e}^r=\beta {\bf e}^r$ for $\beta\neq 0.$ If $D_1,D_2$ are any invertible diagonal matrices, we then have $$\frac{a_{(n+1)r}}{a_{rr}}D_2^{-1}B^{-1}D_1^{-1}{\bf e}^r=\delta {\bf e}^r.$$  Substituting $a_{r(n+1)}=0$ in (\ref{Adagfor}), we infer that $$A^\dag=\begin{pmatrix}
        D_2^{-1}B^{-1}D_1^{-1} & \delta {\bf e}^r\\
        0 & 0
    \end{pmatrix},$$ which is $\Gamma$-treediagonal. \\
    The converse of the second part follows similarly. 
\end{proof}

\begin{ex}
   Let $\Gamma$ be the tree given by
$$
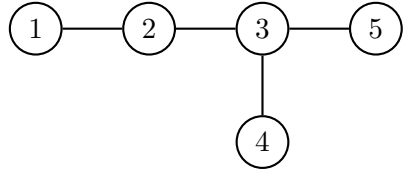
\begin{tikzpicture}[node distance={15mm}, thick, main/.style = {draw, circle}] 
\node[main] (1) {$1$}; 
\node[main] (2) [ right of=1] {$2$}; 
\node[main] (3) [right of=2] {$3$}; 
\node[main] (4) [below of=3] {$4$}; 
\node[main] (5) [right of=3] {$5$}; 
\draw[-](1) -- (2); 
\draw[-] (2) -- (3); 
\draw[-] (3) -- (4); 
\draw[-] (3) -- (5);
\end{tikzpicture}$$
  Here,  $A=\begin{pmatrix}
        1 & -2 & -6 & 0 & -12\\
        1 & ~~1 & ~~3 & 0 & ~~6\\
        0 & ~~0 & ~~1 & 0 & ~~2\\
        0 & ~~0 & -1 & 1 & -2\\
        0 & ~~0 & ~~0 & 0 & ~~0
    \end{pmatrix}$ and $A^\dag=\frac{1}{15}\begin{pmatrix}
      ~~5 & 10 & ~~0 & 0 & 0\\
      -5 & 5 & -45 & 0 & 0\\
      ~~0 & 0 & ~~3 & 0 & 0\\
      ~~0 & 0 & ~~15 & 15 & 0\\
      ~~0 & 0 & ~~6 & 0 & 0
    \end{pmatrix}.$ It may be verified that $A$ satisfies the $\Gamma$-treeangle property, while $A^{\dag}$ is a $\Gamma$-treediagonal matrix. 
\end{ex}

\subsection*{Acknowledgements}
The authors thank the reviewer for his/her diligence. The comments and suggestions have helped in presenting the results in a more lucid manner. 
Andr\'es M. Encinas  has been partially supported by the Spanish
Research Council under project 
PID2021-122501NB-I00 and by the Universitat Polit\`ecnica de Catalunya under funds AGRUPS-UPC  2025.

\end{document}